\documentclass[10pt]{article}

\usepackage[a4paper,margin=1in]{geometry}
\usepackage[T1]{fontenc}
\usepackage[utf8]{inputenc}
\usepackage{amsmath,amssymb,amsthm,mathtools}
\usepackage{enumitem}
\usepackage{microtype}
\usepackage{tikz}
\usepackage{float}
\usepackage[dvipsnames]{xcolor}
\usepackage[colorlinks=true,linkcolor=blue,citecolor=blue,urlcolor=blue]{hyperref}

\numberwithin{equation}{section}

\newtheorem{theorem}{Theorem}[section]
\newtheorem{lemma}[theorem]{Lemma}
\newtheorem{conjecture}[theorem]{Conjecture}
\newtheorem{proposition}[theorem]{Proposition}

\newtheorem{claim}[theorem]{Claim}
\theoremstyle{definition}

\theoremstyle{remark}
\newtheorem{remark}[theorem]{Remark}

\newcommand{\La}{\lambda}
\newcommand{\ex}{\operatorname{ex}}
\newcommand{\cone}{\operatorname{cone}}
\newcommand{\Sint}{\mathbf{S}}

\title{\Large\bf The number $4/9$ is a non-jump for $3$-graphs}
\author{%
Xizhi Liu\thanks{School of Mathematical Sciences, University of Science and Technology of China, Hefei, China. Email:~\texttt{liuxizhi@ustc.edu.cn}.}
\and
Dhruv Mubayi\thanks{Department of Mathematics, Statistics and Computer Science, University of Illinois Chicago, Chicago, IL 60607, USA. Email:~\texttt{mubayi@uic.edu}.}}
\date{\today}

\begin{document}
\maketitle

\begin{abstract}
    We prove that $4/9$ is a non-jump for $3$-uniform hypergraphs. Our construction perturbs the $ABB$ pattern by inserting, inside the $B$-part, the union of a high-cogirth pair of Steiner triple systems.  This goes below the barrier for non-jumps obtainable by Shaw's finite-pattern formulation of the Frankl--R\"odl method introduced in 1984. All results employing this approach use patterns where one of the parts has complete shadow. As the $ABB$ pattern is the smallest one with this property, the value $4/9$ is the natural barrier using this technique, and we conjecture that $4/9$ is the smallest non-jump for $3$-graphs. If our conjecture is true, this would answer (in a very strong form) an old question of Erd\H os.    
\end{abstract}

\section{Introduction}\label{SEC:Introduction}

Given an integer $r\ge 2$, an $r$-uniform hypergraph, or simply an $r$-graph, is a collection of $r$-subsets of a vertex set.  We identify a hypergraph with its edge set and write $|H|$ for the number of its edges. We write $v(H)=|V(H)|$ for the number of vertices in $H$. Given a family $\mathcal F$ of $r$-graphs, let $\ex(n,\mathcal F)$ be the maximum number of edges in an $\mathcal F$-free $r$-graph on $n$ vertices. The \emph{Tur\'an density} of $\mathcal F$ is
\[
        \pi(\mathcal F)
        \coloneqq
        \lim_{n\to\infty}\frac{\ex(n,\mathcal F)}{\binom nr}.
\]
The existence of this limit follows from the averaging argument of Katona, Nemetz and Simonovits~\cite{KatonaNemetzSimonovits64}.  Let $\Pi^{(r)}_\infty$ denote the set of all Tur\'an densities of arbitrary, possibly infinite, families of $r$-graphs. As a convention, we assume that $\mathcal F$ is non-empty, and hence $1\notin \Pi^{(r)}_\infty$.

A number $\alpha\in[0,1)$ is called a \emph{jump} for $r$-graphs if there is a constant $c>0$ such that
\[
        \Pi^{(r)}_\infty\cap(\alpha,\alpha+c)=\emptyset.
\]
Otherwise $\alpha$ is called a \emph{non-jump}.  For graphs, the seminal Erd\H{o}s--Stone theorem~\cite{ErdosStone46} implies that $\Pi^{(2)}_\infty = \{1-1/k \colon k \in \mathbb{N},\ k \ge 2\} \cup \{0\}$, and thus, every $\alpha\in[0,1)$ is a jump.

The jump problem for hypergraphs goes back to a classical work of Erd\H{o}s~\cite{Erdos64}, and became one of his famous $\$1000$ problems. In~\cite{Erdos64}, Erd\H{o}s proved that $\Pi^{(r)}_\infty\cap(0,r!/r^r)=\emptyset$ for every $r\ge3$, and asked whether every number is a jump. Frankl and R\"odl~\cite{FranklRodl84} famously disproved this by proving the existence of non-jumps for every $r\ge3$. Further explicit gaps in $\Pi^{(3)}_\infty$ (that is, jumps) were found by Baber and Talbot~\cite{BaberTalbot11}, using Razborov's flag algebra method~\cite{Raz07} together with a criterion of Frankl and R\"odl.

The topological language of possible Tur\'an densities was developed by Pikhurko~\cite{Pikhurko14}; further results on gaps and on the algebraic and topological structure of possible densities appear in~\cite{BrownSimonovits84,Grosu16,Pikhurko15}.  Codegree and $\ell$-degree versions of Tur\'an density were studied by Mubayi and Zhao~\cite{MubayiZhao07} and by Lo and Markstr\"om~\cite{LoMarkstrom14}; in that setting, it was shown that the non-jumps form a dense set in $[0,1)$.  The multigraph analogue of the jumping constant conjecture was studied by R\"odl and Sidorenko~\cite{RodlSidorenko95}.  Further work on higher-uniformity non-jumps and jumping densities includes works~\cite{Peng07NonJumping4Uniform,Peng09JumpingDensities,Peng08LagrangiansI,Peng07LagrangiansII,HouLiYangZhang24}.  Algebraic aspects of Tur\'an densities were studied by Liu and Pikhurko~\cite{LiuPikhurko23}.  For general background on hypergraph Tur\'an problems, see the excellent survey of Keevash~\cite{Keevash11}.

Many extremal constructions for hypergraph Tur\'an problems are most naturally described by finite patterns, a notation introduced in~\cite{Pikhurko14}.  Informally, a $3$-uniform pattern records only how many vertices an edge takes from each part of a blow-up.  Thus an edge of type $ABB$ represents all triples with one vertex in the $A$-part and two vertices in the $B$-part.  The Lagrangian of a pattern is the maximum asymptotic density of its blow-ups, optimized over all choices of part sizes. Pikhurko~\cite{Pikhurko14} proved that the Lagrangian of every finite pattern is itself a Tur\'an density. For instance, the $ABB$ pattern has normalized Lagrangian
\[
        \max_{\substack{a+b=1\\ a,b\ge0}} 3ab^2=\frac49.
\]

As already shown in the classical work of Frankl and R\"odl~\cite{FranklRodl84}, Lagrangians are useful in jump/non-jump problems because they connect finite local objects with asymptotic densities.  If a finite hypergraph or pattern has Lagrangian greater than $\alpha$, then suitable blow-ups give arbitrarily large hypergraphs of density greater than $\alpha$.  Conversely, if every bounded local subgraph arising in a construction has Lagrangian at most $\alpha$, then no bounded subgraph witnesses density larger than $\alpha$.  The Frankl--R\"odl method exploits this tension: one inserts a sparse structure inside a blow-up part in order to raise the global Lagrangian, while not increasing the Lagrangian of small induced subhypergraphs.

Recently, Shaw~\cite{Shaw2025MinimalNonJumps} formulated the Frankl--R\"odl construction in a finite-pattern language.  In this formulation, one starts with a finite $r$-pattern $P$ and a distinguished vertex $v$, forms an auxiliary Frankl--R\"odl pattern $\mathrm{FR}_v(P)$, and proves the Lagrangian identity $\lambda(\mathrm{FR}_v(P))=\lambda(P)$. 
Very roughly, $\mathrm{FR}_v(P)$ records the worst local configuration created when a sparse $r$-graph is inserted into the part corresponding to $v$ in a large blow-up of $P$. Thus the identity above is the local Lagrangian condition which allows the inserted sparse hypergraph to raise the global density while keeping every bounded local subgraph at the original Lagrangian level.

Shaw proved that this finite-pattern version of the method cannot produce $3$-graph non-jumps below
\[
        \frac{6}{121}\left( 5\sqrt5-2 \right)=0.4552\ldots.
\]
In particular, it cannot reach $4/9$.  He also asked whether $\frac49=\frac{2\cdot 3!}{3^3}$ is a jump for $3$-graphs.  Further examples within the same pattern framework were obtained by Komorech~\cite{Komorech2025NonJumps}.

The value $4/9$ is a natural boundary for this circle of ideas. In all known finite-pattern implementations of the Frankl--R\"{o}dl method for $3$-graphs, the sparse object inserted inside one part is controlled by a repeated-part edge. After relabelling, this leads to the $ABB$ profile.
Since the pure $ABB$ pattern already has normalized Lagrangian $4/9$, the Frankl--R\"odl approach appears to have an intrinsic local barrier at this value: below $4/9$ there is no room for the necessary $ABB$ profile.  The present paper shows that this heuristic boundary is attainable.

\begin{theorem}\label{thm:main}
    The number $4/9$ is a non-jump for $3$-uniform hypergraphs.
\end{theorem}

By a result of Peng~\cite{Peng08LagrangiansI}, non-jumps of the form $\alpha r!/r^r$ lift to all larger uniformities.  Thus Theorem~\ref{thm:main} implies that $2r!/r^r$ is a non-jump for every $r\ge4$, a conclusion also obtained by Shaw~\cite{Shaw2025MinimalNonJumps}.

Our construction starts from the $ABB$ pattern. Recall that an $ABB$-construction is a $3$-graph whose vertex set is split into two parts $A$ and $B$, and whose edges are all triples with one vertex in $A$ and two vertices in $B$.
This is also the extremal construction behind the $F_{3,2}$ ($3$-book with $3$-pages) Tur\'an density $4/9$; see~\cite{MubayiRodl02,FurediPikhurkoSimonovits03}.
To obtain a non-jump at $4/9$, we add triples inside the $B$-part.  The internal triple system must be dense enough to raise the global Lagrangian above $4/9$, but every bounded local piece of the resulting hypergraph must still have Lagrangian at most $4/9$.  We achieve this by using an internal triple system with maximum codegree at most two and high cogirth.  We obtain this from the union of two edge-disjoint Steiner triple systems, while the required local sparsity comes from the high-cogirth condition.  The existence of such high-cogirth pairs follows from a recent theorem of Delcourt and Postle~\cite[Theorem~3.3]{DelcourtPostle24}, extending breakthrough work on high-girth Steiner triple systems~\cite{BW19,GlockKuhnLoOsthus20,KwanSahSawhneySimkin24}.

The main new estimate is a local Lagrangian bound for cones.  Given a $3$-graph $Q$, let $\cone(Q)$ be the hypergraph obtained by adding a new vertex $v_0$, all triples $v_0uv$ with $u,v\in V(Q)$, and all edges of $Q$.  We prove that if $Q$ is sparse (which will be defined in Section~\ref{sec:prelim}) and has maximum codegree at most two, then the Lagrangian of $\cone(Q)$ is at most $4/9$. This cone estimate replaces, in our setting, the finite-pattern identity $\lambda(\mathrm{FR}_v(P))=\lambda(P)$.  It allows the whole construction to go below Shaw's finite-pattern barrier.

We end with the following admittedly bold conjecture about 3-graphs that would imply a solution to Erd\H{o}s' \$1000 question about 2/9 being a jump and much more.

\begin{conjecture} \label{conj:4/9}
   All numbers in $[0,4/9)$ are jumps and all numbers in $[4/9, 1)$ are non-jumps.
   \end{conjecture}
It is worth noting that if the first part of Conjecture~\ref{conj:4/9} is true, then Theorem~\ref{thm:main} would show that $4/9$ is the smallest non-jump, thus marking the end of a long line of research obtaining successively smaller non-jumps for 3-graphs. Figure 1 below indicates the current state of the art (cyan represents jumps and red represents non-jumps). See the concluding remarks section for more details on this figure.

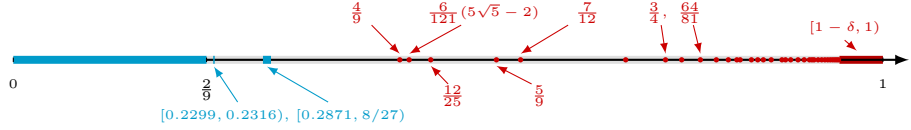
\begin{figure}[H]
\centering
\begin{tikzpicture}[x=11.5cm,y=1cm,>=latex]
\colorlet{jumpcolor}{cyan!80!black}
\colorlet{nonjumpcolor}{red!80!black}
\def\dotr{1pt}
\def\intervallw{3pt}
\def\figlabelsize{\tiny}
\def\upperlabely{0.62}
\def\lowerlabely{-0.48}

\pgfmathsetmacro{\erdos}{2/9}
\pgfmathsetmacro{\our}{4/9}
\pgfmathsetmacro{\shaw}{6*(5*sqrt(5)-2)/121}
\pgfmathsetmacro{\fprt}{5/9}
\pgfmathsetmacro{\pengacc}{7/12}
\pgfmathsetmacro{\ypbase}{12/25}
\pgfmathsetmacro{\komthreeq}{3/4}
\pgfmathsetmacro{\komsixtyfour}{64/81}
\pgfmathsetmacro{\btoneleft}{0.2299}
\pgfmathsetmacro{\btoneright}{0.2316}
\pgfmathsetmacro{\bttwoleft}{0.2871}
\pgfmathsetmacro{\bttworight}{8/27}
\pgfmathsetmacro{\lp}{0.95}

\draw[gray!25,line width=\intervallw] (\erdos,0) -- (\our,0);
\draw[gray!25,line width=\intervallw] (\our,0) -- (\lp,0);
\draw[red!70!black,line width=\intervallw] (\lp,0) -- (1,0);

\draw[thick,->] (0,0) -- (1.03,0);

\draw[jumpcolor,line width=\intervallw] (0,0) -- (\erdos,0);
\draw[jumpcolor,line width=\intervallw] (\btoneleft,0) -- (\btoneright,0);
\draw[jumpcolor,line width=\intervallw] (\bttwoleft,0) -- (\bttworight,0);

\foreach \x/\lab in {
  0/{0},
  \erdos/{$\frac{2}{9}$},
  1/{1}
}{
  \draw (\x,0.03) -- (\x,-0.03);
  \node[below=3pt] at (\x,-0.03) {\figlabelsize \lab};
}

\node[jumpcolor,align=center] at (0.310,-0.74)
  {\figlabelsize $[0.2299,0.2316)$, $[0.2871,8/27)$};
\draw[jumpcolor,->] (0.258,-0.56) -- ({(\btoneleft+\btoneright)/2},-0.09);
\draw[jumpcolor,->] (0.362,-0.56) -- ({(\bttwoleft+\bttworight)/2},-0.09);

\fill[nonjumpcolor] (\our,0) circle (\dotr);
\node[align=center,text=nonjumpcolor] at (0.395,\upperlabely)
  {\figlabelsize $\frac{4}{9}$};
\draw[nonjumpcolor,->] (0.412,0.46) -- (\our,0.08);

\fill[nonjumpcolor] (\shaw,0) circle (\dotr);
\node[align=center,text=nonjumpcolor] at (0.540,\upperlabely)
  {\figlabelsize $\frac{6}{121}(5\sqrt5-2)$};
\draw[nonjumpcolor,->] (0.515,0.46) -- (\shaw,0.08);

\fill[nonjumpcolor] (\fprt,0) circle (\dotr);
\node[align=center,text=nonjumpcolor] at (0.605,\lowerlabely)
  {\figlabelsize $\frac{5}{9}$};
\draw[nonjumpcolor,->] (0.596,-0.32) -- (\fprt,-0.08);
\foreach \fprtl in {15,16,18,20,25,35}{
  \pgfmathsetmacro{\fprtpoint}{1 - 3/\fprtl + 5/(\fprtl*\fprtl)}
  \fill[nonjumpcolor] (\fprtpoint,0) circle (\dotr);
}

\fill[nonjumpcolor] (\pengacc,0) circle (\dotr);
\node[align=center,text=nonjumpcolor] at (0.660,\upperlabely)
  {\figlabelsize $\frac{7}{12}$};
\draw[nonjumpcolor,->] (0.638,0.46) -- (\pengacc,0.08);

\fill[nonjumpcolor] (\ypbase,0) circle (\dotr);
\node[align=center,text=nonjumpcolor] at (0.505,\lowerlabely)
  {\figlabelsize $\frac{12}{25}$};
\draw[nonjumpcolor,->] (0.500,-0.32) -- (\ypbase,-0.08);
\foreach \ypoint in {0.704204,0.768556,0.808429,0.835869,0.856038,0.871551,0.883890,0.893958,0.902343,0.909442,0.915537,0.920831,0.925474,0.929582,0.933243,0.936530,0.939496,0.942188,0.944642,0.946890,0.948956}{
  \fill[nonjumpcolor] (\ypoint,0) circle (\dotr);
}

\fill[nonjumpcolor] (\komthreeq,0) circle (\dotr);
\fill[nonjumpcolor] (\komsixtyfour,0) circle (\dotr);
\node[align=center,text=nonjumpcolor] at (0.760,\upperlabely)
  {\figlabelsize $\frac{3}{4},\ \frac{64}{81}$};
\draw[nonjumpcolor,->] (0.745,0.46) -- (\komthreeq,0.08);
\draw[nonjumpcolor,->] (0.782,0.46) -- (\komsixtyfour,0.08);

\node[align=center,text=nonjumpcolor] at (0.956,0.42)
  {\figlabelsize $[1-\delta,1)$};
\draw[nonjumpcolor,->] (0.956,0.26) -- (0.970,0.12);

\end{tikzpicture}
\caption{Current partial picture for jumps and non-jumps of $3$-graphs.}
\label{fig:numberline-jumps-nonjumps}
\end{figure}

The paper is organized as follows.  Section~\ref{sec:prelim} states the Lagrangian non-jump criterion and the local cone estimate.  Section~\ref{sec:local} proves the cone estimate.  Section~\ref{SEC:proof-main} constructs the required internal triple systems from high-cogirth Steiner triple systems and proves Theorem~\ref{thm:main}. Section~\ref{SEC:remark} contains remarks on the partial picture of known jumps and non-jumps.
The Appendix contains the proof of the standard non-jump criterion and a one-variable analytic lemma used in Section~\ref{sec:local}.

\section{Preliminaries}\label{sec:prelim}

For a finite $3$-graph $H$ and a vector $\mathbf{x}=(x_i)_{i\in V(H)}$, define the \emph{normalized Lagrangian polynomial} of $H$ by
\[
        p_H(\mathbf{x})
        \coloneqq 6\sum_{ijk\in H}x_ix_jx_k.
\]
The \emph{normalized Lagrangian} of $H$ is
\[
        \La(H)
        \coloneqq \max\Big\{p_H(\mathbf{x}):
        x_i\ge0\text{ for every }i\in V(H),\ \sum_{i\in V(H)}x_i=1 \Big\}.
\]
The normalization is chosen so that $\La(H)$ is the asymptotic edge density of optimal blow-ups of $H$, measured relative to $\binom n3$. This is the same normalization as the Lagrangian polynomial in the pattern/blow-up framework; see Motzkin--Straus~\cite{MotzkinStraus65}, Frankl--F\"uredi~\cite{FranklFuredi89}, and Pikhurko~\cite{Pikhurko14}.

We shall use the following standard Lagrangian form of the Frankl--R\"odl method. A proof is included in the Appendix.

\begin{lemma}\label{lem:FR-criterion}
Let $\alpha\in[0,1)$. Suppose that for every integer $m$ there exists a finite $3$-graph $G_m$ such that $\La(G_m)>\alpha$, but $\La(H)\le \alpha$ for every subgraph $H\subseteq G_m$ with $v(H)\le m$. Then $\alpha$ is a non-jump.
\end{lemma}

For distinct vertices $u,v$ of a $3$-graph $Q$, the \emph{degree} of the pair $uv$ is
\[
        d_Q(u,v)
        \coloneqq |\{w \colon uvw\in Q\}|.
\]
The \emph{maximum codegree} of $Q$ is
\[
        \Delta_2(Q)
        \coloneqq \max_{u\ne v}d_Q(u,v).
\]
We call a $3$-graph $Q$ \emph{sparse} if for every $S\subseteq V(Q)$ with $|S|\ge 2$,
\[
        |Q[S]|\le |S|-2.
\]
Given a $3$-graph $Q$, the \emph{cone} of $Q$ is the $3$-graph $\cone(Q)$ given by
\[
        V(\cone(Q))=V(Q)\cup\{v_0\},
        \qquad
        \cone(Q)=Q\cup\{v_0uv \colon u,v\in V(Q),\ u\ne v\},
\]
where $v_0\notin V(Q)$ is called the apex vertex of the cone.

\section{The local Lagrangian estimate}\label{sec:local}

The main result of this section is the following upper bound for the Lagrangian of the cone of a sparse $3$-graph $Q$.

\begin{theorem}\label{thm:local-cone}
Let $Q$ be a sparse $3$-graph with $\Delta_2(Q)\le2$. Then $\La(\cone(Q))\le 4/9$.
\end{theorem}

For the remainder of this section, $Q$ is a finite $3$-graph. For a probability vector $\mathbf{z}=(z_i)_{i\in V(Q)}$, write
\[
        q_Q(\mathbf{z}) \coloneqq \sum_{ijk\in Q}z_iz_jz_k,
        \qquad
        \rho(\mathbf{z}) \coloneqq \sum_{i\in V(Q)}z_i^2.
\]
When $Q$ is clear, write $q=q_Q(\mathbf{z})$ and $\rho=\rho(\mathbf{z})$.

\subsection{Optimizing the apex weight}

Give the apex vertex $v_0$ weight $1-b$, and give the vertices of $Q$ total weight $b$, distributed according to the probability vector $\mathbf{z}$. Thus the weight of $i\in V(Q)$ is $bz_i$. The contribution of the cone edges is
\[
\begin{aligned}
        6\sum_{i<j}(1-b)(bz_i)(bz_j)
        =6(1-b)b^2\sum_{i<j}z_iz_j  
        =3(1-b)b^2(1-\rho),
\end{aligned}
\]
and the contribution of the edges of $Q$ is $b^3p_Q(\mathbf{z})=6b^3q$. Hence the normalized Lagrangian polynomial of $\cone(Q)$ at this weighting is
\begin{align}\label{eq:cone-poly}
    \Phi(b,\mathbf{z})
        &\coloneqq 6\Big( (1-b)b^2\sum_{i<j}z_iz_j
             +b^3q \Big) \notag \\
        &=6(1-b)b^2\frac{1-\rho}{2}+6b^3q
        =3(1-b)b^2(1-\rho)+6b^3q.
\end{align}

Define
\begin{equation}\label{eq:tau-def}
        \tau(\rho) \coloneqq 
        \begin{cases}
        \dfrac{(1-\rho)(1-\sqrt{1-\rho})}{2},&0\le \rho\le 5/9,\\[1.2ex]
        \dfrac{2}{27},&5/9\le \rho\le 1.
        \end{cases}
\end{equation}
For fixed $\rho$, the expression in~\eqref{eq:cone-poly} is increasing in $q$. Thus the relevant quantity is the largest value of $q$ for which optimizing over the apex weight $b$ still gives value at most $4/9$. This critical value is $\tau(\rho)$ defined above. The change at $\rho=5/9$ corresponds to the maximizing value of $b$ moving from an interior point to an endpoint.

The next lemma is a one-variable calculation; its proof is given in the Appendix.

\begin{lemma}\label{lem:tau-threshold}
If $q\le\tau(\rho)$, then
\[
        \max_{0\le b\le1}\Phi(b,\mathbf{z})\le\frac49.
\]
\end{lemma}

Thus Theorem~\ref{thm:local-cone} follows from the following proposition.

\begin{proposition}\label{prop:q-tau}
Let $Q$ be a sparse $3$-graph with $\Delta_2(Q)\le2$. Then, for every probability vector $\mathbf{z}$ on $V(Q)$, we have $q_Q(\mathbf{z})\le \tau(\rho(\mathbf{z}))$.
\end{proposition}

\subsection[Universal 1/27 bound]{A universal $1/27$ bound}

We first prove a bound that uses only sparsity.

\begin{lemma}\label{lem:one-over-27}
Let $Q$ be a sparse $3$-graph. Then, for every probability vector $\mathbf{z}$ on $V(Q)$, we have $q_Q(\mathbf{z})\le\frac1{27}$.
\end{lemma}

\begin{proof}
Suppose the lemma fails and choose a counterexample with positive support $S=\{i:z_i>0\}$ of minimum size. Replacing $Q$ by $Q[S]$, we may assume that all weights are positive.

We first claim that every pair of vertices of $Q$ is contained in some edge. Indeed, if $u,v$ are not contained together in any edge, then, after fixing $z_u+z_v$ and all other weights, $q_Q(\mathbf{z})$ is an affine function of $z_u$. Hence one of the two endpoints $z_u=0$ or $z_v=0$ does not decrease $q_Q(\mathbf{z})$, contradicting the minimality of the support.

Let $N=v(Q)$. Since every pair is contained in an edge, we have $\binom N2\le 3|Q|$. By sparsity, $|Q|\le N-2$. Thus $\binom N2\le 3(N-2)$, so $N=3$ or $N=4$.

If $N=3$, then $Q$ has at most one edge and $q_Q(\mathbf{z})\le z_1z_2z_3\le1/27$. If $N=4$, then $|Q|\le2$. If $|Q|\le1$, the same bound is immediate. If $|Q|=2$, the two triples share a pair, so after relabeling they are $abc$ and $abd$. Hence $q_Q(\mathbf{z})=z_az_b(z_c+z_d)$. Writing $x=z_a+z_b$, AM-GM gives
\[
        q_Q(\mathbf{z})\le\frac{x^2}{4}(1-x)\le\frac1{27}.
\]
This contradiction proves the lemma.
\end{proof}

Let
\begin{equation}\label{eq:rho0}
        \rho_0
        \coloneqq \frac{5-2\sqrt3}{9}.
\end{equation}
Then $\tau(\rho_0)=1/27$. Also $\tau$ is increasing on $[0,5/9]$, since for $s=\sqrt{1-\rho}$ one has
\[
        \tau'(\rho)=\frac{3s-2}{4}\ge0
        \qquad (0\le\rho\le5/9).
\]
Therefore Lemma~\ref{lem:one-over-27} proves Proposition~\ref{prop:q-tau} whenever $\rho\ge\rho_0$. It remains to rule out counterexamples with $\rho<\rho_0$.

\subsection[The low-rho region]{The low-$\rho$ region}

Assume for contradiction that Proposition~\ref{prop:q-tau} fails. By Lemma~\ref{lem:one-over-27}, there are a $3$-graph $Q$ satisfying the assumptions of Proposition~\ref{prop:q-tau} and a probability vector $\mathbf{z}^0$ on $V(Q)$ such that
\[
        \rho(\mathbf{z}^0)<\rho_0,
        \qquad
        q_Q(\mathbf{z}^0)>\tau(\rho(\mathbf{z}^0)).
\]
Fix this $Q$. Choose a probability vector $\mathbf{z}$ maximizing
\[
        F(\mathbf{z})
        \coloneqq q_Q(\mathbf{z})-\tau(\rho(\mathbf{z}))
\]
over the compact set
\[
       \Big\{ \mathbf{z} \in \mathbb{R}^{v(Q)} \colon z_i\ge0 \text{ for } i \in V(Q),\ \sum_{i\in V(Q)}z_i=1,\ \rho(\mathbf{z})\le\rho_0 \Big\}.
\]
The maximum is positive and cannot occur on the boundary $\rho=\rho_0$, because Lemma~\ref{lem:one-over-27} gives $q_Q(\mathbf{z})\le1/27=\tau(\rho_0)$ there.

If some coordinates of this maximizing vector are zero, let $S=\{i:z_i>0\}$ and replace $Q$ by $Q[S]$. The assumptions on $Q$ are hereditary, and the same vector, now viewed as a point in the relative interior of the simplex on $S$, still maximizes $F$ for $Q[S]$; otherwise a better vector on $S$ would also be a better vector for the original $Q$. Thus we may assume that $z_i>0$ for every $i\in V(Q)$. Let $N=v(Q)$.

For each vertex $i$, define its weighted link value by $d_i \coloneqq \sum_{jk:ijk\in Q}z_jz_k$. Let $s \coloneqq \sqrt{1-\rho}$. Since $\rho<\rho_0$, we have
\begin{equation}\label{eq:s-range}
        s>\sqrt{1-\rho_0}=\frac{1+\sqrt3}{3}.
\end{equation}
In this range,
\[
        \tau(\rho)=\frac{s^2(1-s)}2,
        \qquad
        B:=\tau'(\rho)=\frac{3s-2}{4}.
\]
For this fixed maximizing vector, the partial derivatives are $\frac{\partial q_Q}{\partial z_i}=d_i$ and $\frac{\partial \rho}{\partial z_i}=2z_i$. Hence
\[
        \frac{\partial F}{\partial z_i}
        =d_i-\tau'(\rho)\,2z_i
        =d_i-2Bz_i.
\]
Since the maximum is attained in the relative interior of the simplex and away from the boundary $\rho=\rho_0$, the only active constraint is $\sum_i z_i=1$. Its gradient is the all-one vector, so the Lagrange multiplier condition says that all partial derivatives of $F$ are equal to the same constant. Thus there is a real number $\mu$ such that $d_i-2Bz_i=\mu$ for every $i\in V(Q)$. Multiplying by $z_i$ and summing over $i$ gives $\mu=3q-2B\rho$. Since $q>\tau(\rho)$, we obtain
\begin{equation}\label{eq:kkt-lower}
\begin{aligned}
        d_i
        =\mu+2Bz_i
        =3q-2B\rho+2Bz_i 
        >3\tau(\rho)-2B\rho+2Bz_i
        =A+2Bz_i,
\end{aligned}
\end{equation}
where
\[
        A
        \coloneqq 3\tau(\rho)-2B\rho
        = \frac{(1-s)(2-s)}2.
\]
Summing~\eqref{eq:kkt-lower} over all vertices gives
\begin{equation}\label{eq:sum-di-lower}
        \sum_i d_i>NA+2B.
\end{equation}

We now upper-bound the same quantity. For a pair $jk$, let $m_{jk}$ be the codegree of $jk$ in $Q$. The condition $\Delta_2(Q)\le2$ gives $m_{jk}\in\{0,1,2\}$. Since $\mathbf{z}$ has full support and $Q$ is sparse, we have $\sum_{j<k}m_{jk}=3|Q|\le3N-6$. As $m_{jk}^2\le2m_{jk}$, we obtain
\begin{equation}\label{eq:m-square}
        \sum_{j<k}m_{jk}^2\le6N-12.
\end{equation}
Moreover, $\sum_i d_i=\sum_{j<k}m_{jk}z_jz_k$. By Cauchy's inequality,~\eqref{eq:m-square}, and
\[
        \rho^2
        = \Big(\sum_i z_i^2\Big)^2
        \ge 2\sum_{i<j}z_i^2z_j^2,
\]
we get
\begin{align}
        \sum_i d_i
        \le
        \Big( \sum_{j<k}m_{jk}^2 \Big)^{1/2}
        \Big( \sum_{j<k}z_j^2z_k^2 \Big)^{1/2} 
        \le \sqrt{6N-12}\left(\frac{\rho^2}{2}\right)^{1/2}
        =\rho\sqrt{3N-6}.
\label{eq:sum-di-upper}
\end{align}
Combining~\eqref{eq:sum-di-lower} and~\eqref{eq:sum-di-upper}, any counterexample must satisfy
\begin{equation}\label{eq:necessary-counter}
        NA+2B<\rho\sqrt{3N-6}.
\end{equation}

We show that~\eqref{eq:necessary-counter} is impossible. Put $y \coloneqq \sqrt{3N-6}$. Then $N=(y^2+6)/3$, and
\[
        NA+2B-\rho\sqrt{3N-6}
        =\frac A3y^2-\rho y+2A+2B.
\]
This quadratic in $y$ is bounded below, for all real $y$, by $2A+2B-\frac{3\rho^2}{4A}$. Thus it is enough to prove
\begin{equation}\label{eq:main-algebra}
        8A(A+B)>3\rho^2.
\end{equation}
Using
\[
        A=\frac{(1-s)(2-s)}2,
        \qquad
        B=\frac{3s-2}{4},
        \qquad
        \rho=1-s^2,
\]
a direct expansion gives
\[
        8A(A+B)-3\rho^2
        =(1-s)(s^3+10s^2-11s+1).
\]
Since $0<s<1$, it remains to show
\[
        g(s):=s^3+10s^2-11s+1>0.
\]
By~\eqref{eq:s-range}, $s>(1+\sqrt3)/3$. Also
\[
        g'(s)=3s^2+20s-11>0
        \qquad\text{for all }s\ge\frac{1+\sqrt3}{3},
\]
and
\[
        g\left(\frac{1+\sqrt3}{3}\right)
        =\frac{58-33\sqrt3}{27}>0.
\]
Therefore $g(s)>0$, proving~\eqref{eq:main-algebra}. This contradicts~\eqref{eq:necessary-counter}. Hence no low-$\rho$ counterexample exists, and Proposition~\ref{prop:q-tau} follows.

\begin{proof}[Proof of Theorem~\ref{thm:local-cone}]
Consider an arbitrary probability weighting $\mathbf{x}$ of $V(\cone(Q))$. Put
\[
        b \coloneqq \sum_{i\in V(Q)}x_i,
        \qquad
        x_{v_0} \coloneqq 1-b.
\]
If $b=0$, then all weight is on the apex vertex, so the Lagrangian polynomial is zero. Suppose now that $b>0$, and define a probability vector $\mathbf{z}$ on $V(Q)$ by $z_i \coloneqq x_i/b$. 
By Proposition~\ref{prop:q-tau}, we have $q_Q(\mathbf{z})\le\tau(\rho(\mathbf{z}))$. Thus the hypothesis of Lemma~\ref{lem:tau-threshold} is satisfied for the expression in~\eqref{eq:cone-poly}, with $q=q_Q(\mathbf{z})$ and $\rho=\rho(\mathbf{z})$. Since the present value of $b$ lies in $[0,1]$, Lemma~\ref{lem:tau-threshold} gives
\[
        p_{\cone(Q)}(\mathbf{x})
        =
        \Phi(b,\mathbf{z})
        \le
        \max_{0\le b'\le1}\Phi(b',\mathbf{z})
        \le \frac49.
\]
Since this holds for every weighting $\mathbf{x}$, taking the maximum over all weightings gives $\La(\cone(Q))\le4/9$.
\end{proof}

\section{Proof of Theorem~\ref{thm:main}}\label{SEC:proof-main}

We present the proof of Theorem~\ref{thm:main} in this section. We need a $3$-graph on $t$ vertices with maximum codegree at most $2$, more than $t^2/4$ edges, and no small subgraph with too many edges. This is obtained from the high-cogirth design theorem of Delcourt and Postle~\cite{DelcourtPostle24}.

We recall the part of their theorem that we need. For two Steiner triple systems $S_1,S_2$ on the same ground set, regard triples as labelled by the system to which they belong. The pair $(S_1,S_2)$ has cogirth at least $g$ if, for every $2\le i<g$, no $i$ labelled triples from $S_1\sqcup S_2$ span at most $i+1$ vertices.

\begin{theorem}[\cite{DelcourtPostle24}, specialized form]\label{thm:DP}
    For every integer $g \ge 3$ there is $t_0$ such that for every $t\ge t_0$ with $t\equiv1,3\pmod6$, there exist two edge-disjoint Steiner triple systems $S_1,S_2$ on the same $t$-vertex set such that each has girth at least $g$, and the pair $(S_1,S_2)$ has cogirth at least $g$.
\end{theorem}

\begin{remark}
We briefly explain why Theorem~\ref{thm:DP} may be stated with the words
``edge-disjoint''. Delcourt and Postle's high-cogirth design theorem
gives, for every fixed $g$, two $(n,q,r)$-Steiner systems with girth
at least $g$ and cogirth at least $g$. They also describe this result
as producing two disjoint high-girth Steiner systems with high cogirth.

We use only the case $q=3$ and $r=2$, where a
$(t,3,2)$-Steiner system is a Steiner triple system. In this case the
cogirth condition is imposed on the labelled union $S_1\sqcup S_2$ and
says that no $i$ labelled triples, with $2\le i<g$, span at most
$i+1$ vertices. Thus, if $S_1$ and $S_2$ had a common triple
$abc$, the two labelled copies of $abc$ would span only the three
vertices $a,b,c$, giving a forbidden $(3,2)$-configuration whenever
the cogirth is at least $3$. Since Theorem~\ref{thm:DP} is applied only with $g\ge3$, the cogirth condition itself rules out common triples. Thus the two Steiner triple systems may be taken to be edge-disjoint.
\end{remark}

\begin{lemma}\label{lem:code}
For every integer $m$ there are arbitrarily large integers $t$ and a $3$-graph $\Sint$ on $t$ vertices such that
\begin{enumerate}[label=(\roman*)]
\item $\Delta_2(\Sint)\le2$;
\item $|\Sint|=t(t-1)/3$;
\item every induced subgraph of $\Sint$ on at most $m$ vertices is sparse.
\end{enumerate}
\end{lemma}

\begin{remark}
   We note that we do not need $|\Sint|$ to be as large as $t(t-1)/3$ for our construction to work. Indeed, as the proof of Theorem~\ref{thm:main} below shows, all we need is the inequality
   $$ \frac49\left(1-\frac1t\right)+6|\Sint|\left(\frac{2}{3t}\right)^3>\frac49$$
   which is equivalent to $|\Sint|>t^2/4$.
\end{remark}

\begin{proof}
Fix an integer $g>\max\left\{2,\binom m3\right\}$. By Theorem~\ref{thm:DP}, for all sufficiently large $t\equiv1,3\pmod6$ there are two edge-disjoint Steiner triple systems $S_1,S_2$ on the same $t$-vertex set whose pair-cogirth is at least $g$. Let $\Sint \coloneqq S_1\cup S_2$. Since $S_1$ and $S_2$ are edge-disjoint, we have
\[
        |\Sint|=|S_1|+|S_2|=2\cdot\frac{\binom t2}{3}=\frac{t(t-1)}3.
\]
Each pair of vertices lies in exactly one triple of $S_1$ and exactly one triple of $S_2$. Hence each pair lies in two triples of $\Sint$, proving $\Delta_2(\Sint)\le2$.

It remains to prove local sparsity. Suppose that some $U\subseteq V(\Sint)$ with $|U|\le m$ satisfies $|\Sint[U]|>|U|-2$. Let $i \coloneqq |\Sint[U]|$. A single triple spans three vertices, so $i=1$ cannot violate the inequality. Thus $i\ge2$. The $i$ triples of $\Sint[U]$ span at most $|U|\le i+1$ vertices, so they form an $(i+1,i)$-configuration. Also
\[
        i=|\Sint[U]|\le \binom{|U|}{3}\le\binom m3<g.
\]
This contradicts the cogirth condition. Hence no such $U$ exists.
\end{proof}

We now prove Theorem~\ref{thm:main}.

\begin{proof}[Proof of Theorem~\ref{thm:main}]
Fix an integer $m$. Let $\Sint$ be given by Lemma~\ref{lem:code} on a vertex set $B$ of size $t$, where $t$ is large. Define $G \coloneqq \cone(\Sint)$.

\begin{claim}\label{lem:global-lag}
For all sufficiently large $t$, we have $\La(G)>\frac49$.
\end{claim}

\begin{proof}
Let $x_{v_0} \coloneqq \frac13$ and $x_i \coloneqq \frac{2}{3t}$ for $i \in B$. The $v_0BB$-edges contribute
\[
        6\cdot\frac13\binom t2\left(\frac{2}{3t}\right)^2
        =\frac49\left(1-\frac1t\right).
\]
The edges inside $B$ contribute, using $|\Sint|=t(t-1)/3$,
\[
        6|\Sint|\left(\frac{2}{3t}\right)^3
        =6\cdot\frac{t(t-1)}3\cdot\frac{8}{27t^3}
        =\frac{16(t-1)}{27t^2}.
\]
Therefore
\[
        \La(G)
        \ge
        \frac49\left(1-\frac1t\right)+\frac{16(t-1)}{27t^2}
        =\frac49+\frac4{27t}-\frac{16}{27t^2}.
\]
This is greater than $4/9$ whenever $t>4$.
\end{proof}

\begin{claim}\label{lem:small-subgraphs}
Every subgraph $H\subseteq G$ with $v(H)\le m$ satisfies $\La(H)\le \frac49$.
\end{claim}

\begin{proof}
Let $U \coloneqq V(H)\cap B$. Then $|U|\le m$. By Lemma~\ref{lem:code}, the induced subgraph $\Sint[U]$ is sparse and satisfies $\Delta_2(\Sint[U])\le2$.

If $v_0\in V(H)$, then $H\subseteq \cone(\Sint[U])$. If $v_0\notin V(H)$, then $H\subseteq \Sint[U]$, and after adding an unused apex vertex we again have $H\subseteq \cone(\Sint[U])$. By Theorem~\ref{thm:local-cone}, $\La(\cone(\Sint[U]))\le\frac49$. Since normalized Lagrangian is monotone under adding edges and isolated vertices, it follows that $\La(H)\le4/9$.
\end{proof}

For every integer $m$, Claims~\ref{lem:global-lag} and~\ref{lem:small-subgraphs} give a finite $3$-graph $G$ such that $\La(G)>\frac49$ and $\La(H)\le\frac49$ for every subgraph $H\subseteq G$ with $v(H)\le m$. Lemma~\ref{lem:FR-criterion}, applied with $\alpha=4/9$, shows that $4/9$ is a non-jump.
\end{proof}
\section{Concluding remarks}\label{SEC:remark}
Figure~\ref{fig:numberline-jumps-nonjumps-2} gives a partial picture for jumps and non-jumps of $3$-graphs. The cyan intervals are known jumps, the red points and interval are parts of known non-jumps, and the gray portions indicate regions where the picture is largely open.

On the jump side, Erd\H{o}s~\cite{Erdos64} proved that every number in $[0,2/9)$ is a jump, and Baber--Talbot~\cite{BaberTalbot11} proved the jump intervals $[0.2299,0.2316)$ and $[0.2871,8/27)$.

On the non-jump side, the figure marks $4/9$ from this paper, $6(5\sqrt5-2)/121$ from Shaw~\cite{Shaw2025MinimalNonJumps}, $5/9$ and selected values (part of the sequence $\beta_\ell=1-3/\ell+5/\ell^2$) from Frankl--Peng--R\"odl--Talbot~\cite{FranklPengRodlTalbot07}, $12/25$ from Yan--Peng~\cite{YanPeng23}, and the values $3/4$ and $64/81$ from Komorech~\cite{Komorech2025NonJumps}. It also marks selected values from the original Frankl--R\"odl~\cite{FranklRodl84} sequence $1-\frac1{k^2}$ for $k>6$, together with an interval $[1-\delta,1)$ from Liu--Pikhurko~\cite{LiuPikhurkoIntervals}.
The selected Yan--Peng~\cite{YanPeng23} values beyond $12/25$ come from the sequence
\[
        \alpha_k=
        \frac{2k-6k^3+4k^4-k\sqrt{4k-1}+4k^2\sqrt{4k-1}}{(2k^2+1)^2},
        \qquad k\ge2,
\]
which consists of non-jumps for $3$-graphs. Finally, the point $7/12$ is the accumulation point of Peng's~\cite{Peng09Substructure} sequence of non-jumps.

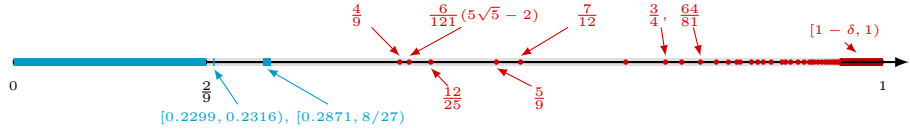
\begin{figure}[H]
\centering
\begin{tikzpicture}[x=11.5cm,y=1cm,>=latex]
\colorlet{jumpcolor}{cyan!80!black}
\colorlet{nonjumpcolor}{red!80!black}
\def\dotr{1pt}
\def\intervallw{3pt}
\def\figlabelsize{\tiny}
\def\upperlabely{0.62}
\def\lowerlabely{-0.48}

\pgfmathsetmacro{\erdos}{2/9}
\pgfmathsetmacro{\our}{4/9}
\pgfmathsetmacro{\shaw}{6*(5*sqrt(5)-2)/121}
\pgfmathsetmacro{\fprt}{5/9}
\pgfmathsetmacro{\pengacc}{7/12}
\pgfmathsetmacro{\ypbase}{12/25}
\pgfmathsetmacro{\komthreeq}{3/4}
\pgfmathsetmacro{\komsixtyfour}{64/81}
\pgfmathsetmacro{\btoneleft}{0.2299}
\pgfmathsetmacro{\btoneright}{0.2316}
\pgfmathsetmacro{\bttwoleft}{0.2871}
\pgfmathsetmacro{\bttworight}{8/27}
\pgfmathsetmacro{\lp}{0.95}

\draw[gray!25,line width=\intervallw] (\erdos,0) -- (\our,0);
\draw[gray!25,line width=\intervallw] (\our,0) -- (\lp,0);
\draw[red!70!black,line width=\intervallw] (\lp,0) -- (1,0);

\draw[thick,->] (0,0) -- (1.03,0);

\draw[jumpcolor,line width=\intervallw] (0,0) -- (\erdos,0);
\draw[jumpcolor,line width=\intervallw] (\btoneleft,0) -- (\btoneright,0);
\draw[jumpcolor,line width=\intervallw] (\bttwoleft,0) -- (\bttworight,0);

\foreach \x/\lab in {
  0/{0},
  \erdos/{$\frac{2}{9}$},
  1/{1}
}{
  \draw (\x,0.03) -- (\x,-0.03);
  \node[below=3pt] at (\x,-0.03) {\figlabelsize \lab};
}

\node[jumpcolor,align=center] at (0.310,-0.74)
  {\figlabelsize $[0.2299,0.2316)$, $[0.2871,8/27)$};
\draw[jumpcolor,->] (0.258,-0.56) -- ({(\btoneleft+\btoneright)/2},-0.09);
\draw[jumpcolor,->] (0.362,-0.56) -- ({(\bttwoleft+\bttworight)/2},-0.09);

\fill[nonjumpcolor] (\our,0) circle (\dotr);
\node[align=center,text=nonjumpcolor] at (0.395,\upperlabely)
  {\figlabelsize $\frac{4}{9}$};
\draw[nonjumpcolor,->] (0.412,0.46) -- (\our,0.08);

\fill[nonjumpcolor] (\shaw,0) circle (\dotr);
\node[align=center,text=nonjumpcolor] at (0.540,\upperlabely)
  {\figlabelsize $\frac{6}{121}(5\sqrt5-2)$};
\draw[nonjumpcolor,->] (0.515,0.46) -- (\shaw,0.08);

\fill[nonjumpcolor] (\fprt,0) circle (\dotr);
\node[align=center,text=nonjumpcolor] at (0.605,\lowerlabely)
  {\figlabelsize $\frac{5}{9}$};
\draw[nonjumpcolor,->] (0.596,-0.32) -- (\fprt,-0.08);
\foreach \fprtl in {15,16,18,20,25,35}{
  \pgfmathsetmacro{\fprtpoint}{1 - 3/\fprtl + 5/(\fprtl*\fprtl)}
  \fill[nonjumpcolor] (\fprtpoint,0) circle (\dotr);
}

\fill[nonjumpcolor] (\pengacc,0) circle (\dotr);
\node[align=center,text=nonjumpcolor] at (0.660,\upperlabely)
  {\figlabelsize $\frac{7}{12}$};
\draw[nonjumpcolor,->] (0.638,0.46) -- (\pengacc,0.08);

\fill[nonjumpcolor] (\ypbase,0) circle (\dotr);
\node[align=center,text=nonjumpcolor] at (0.505,\lowerlabely)
  {\figlabelsize $\frac{12}{25}$};
\draw[nonjumpcolor,->] (0.500,-0.32) -- (\ypbase,-0.08);
\foreach \ypoint in {0.704204,0.768556,0.808429,0.835869,0.856038,0.871551,0.883890,0.893958,0.902343,0.909442,0.915537,0.920831,0.925474,0.929582,0.933243,0.936530,0.939496,0.942188,0.944642,0.946890,0.948956}{
  \fill[nonjumpcolor] (\ypoint,0) circle (\dotr);
}

\fill[nonjumpcolor] (\komthreeq,0) circle (\dotr);
\fill[nonjumpcolor] (\komsixtyfour,0) circle (\dotr);
\node[align=center,text=nonjumpcolor] at (0.760,\upperlabely)
  {\figlabelsize $\frac{3}{4},\ \frac{64}{81}$};
\draw[nonjumpcolor,->] (0.745,0.46) -- (\komthreeq,0.08);
\draw[nonjumpcolor,->] (0.782,0.46) -- (\komsixtyfour,0.08);

\node[align=center,text=nonjumpcolor] at (0.956,0.42)
  {\figlabelsize $[1-\delta,1)$};
\draw[nonjumpcolor,->] (0.956,0.26) -- (0.970,0.12);

\end{tikzpicture}
\caption{Current partial picture for jumps and non-jumps of $3$-graphs.}
\label{fig:numberline-jumps-nonjumps-2}
\end{figure}
\section*{Acknowledgments}

Xizhi Liu was supported by the Excellent Young Talents Program (Overseas) of the National Natural Science Foundation of China. Dhruv Mubayi's research was partially supported by NSF Award DMS-2153576.

\section*{Declaration on the use of generative AI}

All of the conceptual ideas for the construction in this paper were obtained by the authors in early 2020. The two missing ingredients were a construction of large girth designs, provided recently via~\cite{DelcourtPostle24}, and a computation of the Lagrangian (Theorem~\ref{thm:local-cone}). The latter was provided by generative AI tools and checked by the authors.

\bibliographystyle{abbrv}
\bibliography{Turan}
\appendix
\section*{Auxiliary proofs}

\begin{proof}[Proof of Lemma~\ref{lem:FR-criterion}]
We use the standard local formulation of jumps, due to Erd\H{o}s and used by Frankl and R\"odl~\cite{FranklRodl84}.

Suppose for a contradiction that $\alpha$ is a jump. In the standard local form of the jump property, this means that there exists $c>0$ such that, for every $\varepsilon>0$ and every integer $M\ge3$, every sufficiently large $3$-graph of density at least $\alpha+\varepsilon$ contains an $M$-vertex subgraph of density at least $\alpha+c$.

Choose an integer $M\ge3$ large enough that
\[
        \alpha\frac{M^2}{(M-1)(M-2)}<\alpha+c.
\]
By the hypothesis, there is a finite $3$-graph $G_M$ such that $\La(G_M)>\alpha$ and every subgraph of $G_M$ on at most $M$ vertices has normalized Lagrangian at most $\alpha$. Choose $\eta>0$ such that
\[
        \alpha+2\eta<\La(G_M).
\]
By the definition of normalized Lagrangian, there is a probability vector $\mathbf{x}=(x_v)_{v\in V(G_M)}$ such that
\[
        p_{G_M}(\mathbf{x})>\alpha+2\eta.
\]
Taking a sufficiently large blow-up of $G_M$ with part sizes asymptotic to the weights $x_v$, we obtain a $3$-graph $B$ of density at least $\alpha+\eta$. Indeed, if the blow-up has part sizes $n_v$ with $n_v/N\to x_v$, then its edge density tends to $p_{G_M}(\mathbf{x})$.

By the jump property, provided $B$ is sufficiently large, it contains a set $W$ of $M$ vertices such that
\[
        |B[W]| \ge (\alpha+c)\binom M3.
\]
Let $V_i$ be the blow-up part corresponding to $i\in V(G_M)$, and set
\[
        I \coloneqq \left\{i\in V(G_M) \colon W\cap V_i\ne\emptyset \right\},
        \qquad
        m_i \coloneqq |W\cap V_i|.
\]
Put $\mathbf{y}=(y_i)_{i\in I}$, where $y_i=m_i/M$. Then $\sum_{i\in I}y_i=1$, and
\[
        |B[W]|
        \le
        \sum_{ijk\in G_M[I]}m_i m_j m_k
        =
        \frac{M^3}{6}p_{G_M[I]}(\mathbf{y})
        \le
        \frac{M^3}{6}\La(G_M[I])
        \le
        \frac{\alpha M^3}{6}.
\]
Dividing by $\binom M3=M(M-1)(M-2)/6$, we obtain
\[
        \frac{|B[W]|}{\binom M3}
        \le
        \alpha\frac{M^2}{(M-1)(M-2)}
        <\alpha+c,
\]
contradicting the choice of $W$. Hence $\alpha$ is a non-jump.
\end{proof}

\begin{proof}[Proof of Lemma~\ref{lem:tau-threshold}]
Recall that $\Phi(b,\mathbf{z})=3(1-b)b^2(1-\rho)+6b^3q$. Put $c \coloneqq 1-\rho$. Then
\[
        \Phi(b,\mathbf{z})=3cb^2(1-b)+6qb^3.
\]

First suppose that $\rho\le5/9$, so $c\ge4/9$. In this range, since $c\ge4/9$ implies $\sqrt c\ge2/3$, we have $\tau(\rho)=\frac{c(1-\sqrt c)}2\le\frac c6$. Hence $q\le c/6$. Let
\[
        f(b) \coloneqq 3cb^2(1-b)+6qb^3.
\]
Then
\[
        f'(b)=3b\bigl(2c-3(c-2q)b\bigr).
\]
Since $q\le c/6$, the critical point $b_0=\frac{2c}{3(c-2q)}$ lies in $[0,1]$. Thus the maximum of $f$ on $[0,1]$ is attained at $b_0$, and a direct substitution gives
\[
        f(b_0)=\frac{4c^3}{9(c-2q)^2}.
\]
Therefore $f(b_0)\le4/9$ is equivalent to $c^{3/2}\le c-2q$,  or equivalently
\[
        q\le\frac{c(1-\sqrt c)}2=\tau(\rho).
\]
This holds by hypothesis.

Now suppose that $\rho\ge5/9$, so $c\le4/9$. In this range $q\le\tau(\rho)=2/27$. If $q\le c/6$, then
\[
        \Phi(b,\mathbf{z})
        \le 3cb^2(1-b)+cb^3
        =c(3b^2-2b^3)
        \le c
        \le\frac49.
\]
It remains to consider the case $q\ge c/6$. Note that $\Phi'(b)=3b\bigl(2c-3(c-2q)b\bigr)$. If $c-2q\ge0$, then for $0\le b\le1$, we have
\[
        2c-3(c-2q)b\ge 2c-3(c-2q)=6q-c\ge0.
\]
If $c-2q<0$, then
\[
        2c-3(c-2q)b\ge 2c\ge0.
\]
Thus $\Phi'(b)\ge0$ on $[0,1]$. Hence $\Phi(b,\mathbf{z})$ is increasing on $[0,1]$, and therefore
\[
        \Phi(b,\mathbf{z})\le\Phi(1,\mathbf{z})=6q\le\frac49.
\]
This proves the lemma.
\end{proof}

\end{document}